\documentclass[11pt, a4paper]{article}

\usepackage{amsmath}
\usepackage{amssymb}
\usepackage{amsthm}
\usepackage{appendix,xcolor}
\usepackage{verbatim}

\newtheorem{theorem}{Theorem}[section]
\newtheorem{corollary}{Corollary}[section]

\newtheorem{definition}{Definition}[section]
\newtheorem{lemma}{Lemma}[section]
\newtheorem{remark}{Remark}[section]
\newtheorem{example}{Example}[section]
\newtheorem*{thmMarkus}{Theorem A}

\begin{document}
\title{On Strong Markushevich bases $\{t^{\lambda_n}\}_{n=1}^{\infty}$
in their closed span in $L^2 (0, 1)$ and characterizing a subspace of $H^2 (\mathbb{D})$}
\author{Elias Zikkos\\Department of Mathematics\\
Khalifa University of Science and Technology, Abu Dhabi, United Arab Emirates\\
email address:  elias.zikkos@ku.ac.ae and eliaszikkos@yahoo.com}

\maketitle

\begin{abstract}
Let $\Lambda=\{\lambda_n\}_{n=1}^{\infty}$ be a strictly increasing sequence of positive real numbers
such that $\sum_{n=1}^{\infty}\frac{1}{\lambda_n}<\infty$ and $\inf(\lambda_{n+1}-\lambda_n)>0$. We investigate properties of the closed span
of the system $\{t^{\lambda_n}\}_{n=1}^{\infty}$ in $L^2 (0,1)$, denoted by $\overline{M_{\Lambda}}$, and of the unique biorthogonal family $\{r_n (t)\}_{n=1}^{\infty}$ to the system $\{t^{\lambda_n}\}_{n=1}^{\infty}$ in $\overline{M_{\Lambda}}$.
We show that the system $\{t^{\lambda_n}\}_{n=1}^{\infty}$ is a strong Markushevich basis in $\overline{M_{\Lambda}}$
and we obtain a series representation for functions in $\overline{M_{\Lambda}}$.
We also construct a general class of operators on $\overline{M_{\Lambda}}$ that admit spectral synthesis. In particular,
for all $\rho \in (0,1)$ the operator $T_{\rho}(f)=f(\rho x)$ on $\overline{M_{\Lambda}}$ admits spectral synthesis.
In addition, we characterize a certain subspace of the classical Hardy space $H^2 (\mathbb{D})$.
Under the extra assumption that $\Lambda\subset\mathbb{N}$, let $H^2(\mathbb{D}, \Lambda)$ consist of functions $f$ in $H^2(\mathbb{D})$
so that the Fourier coefficients $c_n$ of the boundary function $f(e^{i\theta})$ vanish for all $n\notin \Lambda$.
We prove that $f\in H^2(\mathbb{D}, \Lambda)$ if and only if $f\in\overline{M_{\Lambda}}$ and $\sum_{n=1}^{\infty}\left| \langle f, r_n\rangle \right|^2<\infty$, where $\langle f, g\rangle= \int_{0}^{1} f(t)\cdot \overline{g(t)}\, dt$.
\end{abstract}



\section{Introduction}
\setcounter{equation}{0}

Let $\Lambda=\{\lambda_n\}_{n=1}^{\infty}$ be a strictly increasing sequence of positive real numbers such that
\begin{equation}\label{LKcondition}
\sum_{n=1}^{\infty}\frac{1}{\lambda_n}<\infty\qquad\text{and}\qquad \inf(\lambda_{n+1}-\lambda_n)>0.
\end{equation}
We denote by $M_{\Lambda}$ the system $\{e_n(t) : = t^{\lambda_n}\}_{n=1}^{\infty}$ and by $\overline{M_{\Lambda}}$ the closed span of
$M_{\Lambda}$ in $L^2 (0,1)$. Due to $(\ref{LKcondition})$, it follows from the M\"{u}ntz-Sz\'{a}sz theorem that $\overline{M_{\Lambda}}$
is a proper subspace of $L^2 (0,1)$.  In fact, for $f\in \overline{M_{\Lambda}}$ we have the
``Clarkson-Erd\H{o}s-Schwartz Phenomenon'' as follows :

``let $\mathbb{D}:=\{z:\, |z|<1\}$ be the unit disk. Then $f$ extends as an analytic function
in the slit disk $\mathbb{D}^*:=\mathbb{D}\setminus (-1,0)$, so that
\[
f(z)=\sum_{n=1}^{\infty}c_n z^{\lambda_n}\qquad z\in \mathbb{D}^*
\]
with the series converging uniformly on compact subsets of $\mathbb{D}^*$."

(see \cite{CE}, \cite[Corollary 6.2.4]{Gurariy}, \cite[Theorem 6.4]{BE1997JAMS}).

\begin{remark}
One may consult \cite{BE,Gurariy,Almira} on M\"{u}ntz-Sz\'{a}sz topics as well as
\cite{Chalendar,Noor,Lefevre2,Fricain,Agler1,Agler2} on more recent related results.
\end{remark}

Under the condition $(\ref{LKcondition})$, it is also known that $M_{\Lambda}$ is a minimal system in $\overline{M_{\Lambda}}$:
every element of $M_{\Lambda}$ does not belong to the closed span of the remaining elements in $L^2 (0,1)$
(see \cite[Proposition 6.1.4]{Gurariy}). Now, every minimal system in $L^2 (0,1)$ has a biorthogonal family,
that is there exists a system of functions $r_{\Lambda}:=\{r_n(t)\}_{n=1}^{\infty}\subset L^2 (0,1)$ such that
\[
\langle e_n, r_m\rangle =\begin{cases} 1, & m=n, \\  0, & m\not=n,\end{cases}\qquad \text{where}\qquad
\langle f, g\rangle := \int_{0}^{1} f(t)\cdot \overline{g(t)}\, dt,\quad f, \, g \in L^2 (0,1).
\]
Since the system $M_{\Lambda}$ is complete and minimal in $\overline{M_{\Lambda}}$, then $M_{\Lambda}$ has a $unique$ biorthogonal family
$r_{\Lambda}\subset \overline{M_{\Lambda}}$.

We aim to study certain properties of $\overline{M_{\Lambda}}$. We first show in Lemma $\ref{Bio}$ how the unique
biorthogonal $r_{\Lambda}$ family is constructed and obtain the sharp upper bound $(\ref{rnkbound})$ for the norm of its elements.
Building on the ``Clarkson-Erd\H{o}s-Schwartz Phenomenon'', we then prove in Theorem $\ref{biorthogonalsystem}$ that every function in $\overline{M_{\Lambda}}$ admits the series representation $(\ref{representationf})$.
We then show that the closed span of the family $r_{\Lambda}$
in $L^2 (0,1)$ is equal to the space $\overline{M_{\Lambda}}$, hence $M_{\Lambda}$ and $r_{\Lambda}$ are $Markushevich\,\, Bases$ for $\overline{M_{\Lambda}}$. In fact, they are also
\[
Strong\,\, Markushevich\,\, Bases\qquad\text{for}\qquad \overline{M_{\Lambda}}.
\]
This means that for any disjoint union of two sets $N_1$ and $N_2$, where
$\mathbb{N}=N_1\cup N_2$, the closed span of the mixed system
\[
\{e_n:\,\, n\in N_1\}\cup\{r_{n}:\,\, n\in N_2\}
\]
in $L^2 (0,1)$ is equal to the space $\overline{M_{\Lambda}}$.
The phrase $``hereditarily\,\, complete\,\, system''$ is also commonly used instead of ``strong Markushevich basis''.

\begin{theorem}\label{biorthogonalsystem}
Let $\Lambda=\{\lambda_n\}_{n=1}^{\infty}$ be a strictly increasing sequence of positive real numbers
that satisfies $(\ref{LKcondition})$. Let $r_{\Lambda}$ be the unique biorthogonal family of $M_{\Lambda}$ in the space $\overline{M_{\Lambda}}$.
Then the following are true:

$(I)$  Each function $f\in \overline{M_{\Lambda}}$ extends as an analytic function in the slit disk
$\mathbb{D}^*:=\mathbb{D}\setminus (-1,0)$, so that

\begin{equation}\label{representationf}
f(z)=\sum_{n=1}^{\infty}\langle f, r_{n} \rangle z^{\lambda_n},
\end{equation}
converging uniformly on compact subsets of $\mathbb{D}^*$.

$(II)$ For each function $f\in L^2 (0, 1)$ its associated series
\[
f^*(x):=\sum_{n=1}^{\infty} \langle f, r_{n} \rangle x^{\lambda_n}
\]
is analytic in $\mathbb{D}^*$ and belongs to the space $\overline{M_{\Lambda}}$.

$(III)$ The system $M_{\Lambda}$ is a Markushevich basis for $\overline{M_{\Lambda}}$.

$(IV)$ The system $M_{\Lambda}$ is a strong Markushevich basis for $\overline{M_{\Lambda}}$.
\end{theorem}

\begin{remark}
To this end we note that Fricain and Lef\`{e}vre  (see \cite[Corollary 3.9]{Fricain}) proved  by a different method
that the system $M_{\Lambda}$ is a Markushevich basis for $\overline{M_{\Lambda}}$.
\end{remark}
In addition, we point out that in \cite[pages 99-112]{Mash} A. Boivin and C. Zhu in their study
\[
\text{Bi-orthogonal Expansions in the Space $L^2 (0, \infty)$}
\]
explore a similar problem to ours where the exponential system $\{e^{-\lambda_n t}\}_{n=1}^{\infty}$
and the space $L^2 (0, \infty)$ are considered instead of $\{t^{\lambda_n}\}_{n=1}^{\infty}$
and $L^2 (0, 1)$. A biorthogonal family is introduced on page $100$ and eventually the authors prove in
\cite[Corollary 1, page 109]{Mash} that the closed spans of the exponential system and its biorthogonal family in  $L^2 (0, \infty)$
coincide. Thus the exponential family is a Markushevich basis for that closure.

Now, by Theorem $\ref{biorthogonalsystem}$ if a sequence $\Lambda$ satisfies $(\ref{LKcondition})$ then the family $M_{\Lambda}$ is a strong Markushevich basis for the space $\overline{M_{\Lambda}}$.
\[
\text{But is $M_{\Lambda}$ a $basis$ for $\overline{M_{\Lambda}}$?}
\]
The answer is $\bf Yes$ if and only if $\Lambda$ is a lacunary sequence (see \cite[Theorem 9.2.2]{Gurariy}),
meaning that there is some $q>1$ so that for all $n\in\mathbb{N}$ one has $\frac{\lambda_{n+1}}{\lambda_n}>q$.
The same conclusion was drawn by Boivin and Zhu in \cite[Theorem 3, page 112]{Mash}.

\begin{remark}
In \cite{Zikkos2024JMAA} we proved that if $\Lambda$ satisfies the condition $(\ref{LKcondition})$,
then the exponential system $\{e^{\lambda_n t}\}_{n=1}^{\infty}$ is hereditarily complete in its closed span in $L^2 (a,b)$
when $(a, b)$ is a $\bf bounded$ interval on the real line. The proof of Theorem $\ref{biorthogonalsystem}$ in this paper
follows closely the scheme of proving \cite[Lemma 3.1 and Theorem 1.1]{Zikkos2024JMAA}.
\end{remark}

Hereditary completeness (strong M basis) is closely related to the $Spectral\,\, Synthesis$ problem for linear operators
\cite{Markus1970,Baranov2013,Baranov2015}.
Consider a bounded linear operator $T$ in a separable Hilbert space $\cal H$ which has a set of eigenvectors
which is complete in $\cal H$. We say that $T$ admits $Spectral\,\, Synthesis$ if for any invariant subspace $A$ of $T$,
the set of eigenvectors of $T$ contained in $A$ is complete in $A$.
J. Wermer proved in \cite{Wermer1952} that an operator which is compact and normal
admits spectral synthesis. Wermer also gave an example of a bounded normal operator which does not admit spectral synthesis \cite[Theorem 2]{Wermer1952}. We also point out that A. S. Markus proved that there exist compact operators which do not admit spectral synthesis
(see \cite[Theorem 4.2]{Markus1970}).

Markus also proved the following result.
\begin{thmMarkus}\cite[Theorem 4.1]{Markus1970}\label{Compact}
Let $\cal{H}$ be a separable Hilbert space and let $T:\cal{H}\to\cal{H}$ be a compact operator such that

(i) its kernel is trivial and

(ii) its non-zero eigenvalues are simple.

Let $\{f_n\}_{n\in\mathbb{N}}$ be the corresponding sequence of eigenvectors.
Then $T$ admits Spectral Synthesis if and only if $\{f_n\}_{n\in\mathbb{N}}$ is hereditarily complete in $\cal{H}$.
\end{thmMarkus}

Motivated by the Markus result and since $M_{\Lambda}$ is a strong Markushevich basis for $\overline{M_{\Lambda}}$,
we now present a general class of compact but not normal operators on $\overline{M_{\Lambda}}$ that admit spectral synthesis.
A particular example is the operator $T_{\rho}(f):=f(\rho x)$ for $\rho\in (0, 1)$.

\begin{theorem}\label{Spectral Synthesis}
Let $\Lambda=\{\lambda_n\}_{n=1}^{\infty}$ be a strictly increasing sequence of positive real numbers
that satisfies $(\ref{LKcondition})$. Let $r_{\Lambda}$ be the unique biorthogonal family of $M_{\Lambda}$ in the space $\overline{M_{\Lambda}}$.
Fix a sequence $\{u_n\}_{n=1}^{\infty}$ of distinct non-zero complex numbers such that
\begin{equation}\label{un}
|u_n|\le  \rho^{\lambda_n}\qquad \text{for\,\, some}\qquad 0<\rho<1.
\end{equation}
Then the operator $T: \overline{M_{\Lambda}}\to \overline{M_{\Lambda}}$ defined as
\begin{equation}\label{operator}
T(f(x)):=\sum_{n=1}^{\infty} \langle f , r_n\rangle\cdot u_n\cdot  x^{\lambda_n},
\end{equation}
is compact, not normal, and admits Spectral Synthesis.
\end{theorem}

\begin{corollary}
Let $\Lambda=\{\lambda_n\}_{n=1}^{\infty}$ be a strictly increasing sequence of positive real numbers
that satisfies $(\ref{LKcondition})$. Let $r_{\Lambda}$ be the unique biorthogonal family of $M_{\Lambda}$ in the space $\overline{M_{\Lambda}}$.
Then for any fixed $0<\rho<1$,
the operator $T: \overline{M_{\Lambda}}\to \overline{M_{\Lambda}}$, where $T_{\rho}(f):=f(\rho x)$,
is compact, not normal, and admits Spectral Synthesis.
\end{corollary}
\begin{proof}
From $(\ref{representationf})$ if $f\in \overline{M_{\Lambda}}$ then
$f(x)=\sum_{n=1}^{\infty}\langle f, r_{n} \rangle x^{\lambda_n}$, converging uniformly on compact subsets of the interval $[0,1)$. Then
$f(\rho x)=\sum_{n=1}^{\infty}\langle f, r_{n} \rangle\cdot  \rho^{\lambda_n}\cdot x^{\lambda_n}$ and the conclusion follows from the above theorem.
\end{proof}

Our second goal has to do with a characterization of a certain subspace of the classical Hardy space $H^2(\mathbb{D})$.
It is well known that an analytic function $f(z)=\sum_{n=0}^{\infty} a_n z^n$ in
$\mathbb{D}$ belongs to $H^2(\mathbb{D})$
if and only if $\sum_{n=0}^{\infty} |a_n|^2<\infty$. It is also known that
the Taylor coefficients $a_n$ are equal to the Fourier coefficients
\[
c_n=\frac{1}{2\pi}\int_{0}^{2\pi} f(e^{i\theta})\cdot e^{-i n\theta}\, d\theta,\qquad n=0, 1, 2, \dots
\]
of the boundary function $f(e^{i\theta})$, and that the Fourier coefficients $c_n$ vanish for negative integers $n$ (\cite[Theorem 3.4]{Duren}).
\begin{remark}
For various properties of functions in $H^p(\mathbb{D})$ spaces one may consult \cite{Duren,Jevtic,Koosis,Mashreghi}.
\end{remark}
For a sequence $\Lambda=\{\lambda_n\}_{n=1}^{\infty}$ which is a subset of the integer set $\mathbb{N}$ and satisfying $(\ref{LKcondition})$,
we let $H^2 (\mathbb{D}, \Lambda)$ consist of all the functions $f\in H^2 (\mathbb{D})$ such that their
Fourier coefficients $c_n$ vanish whenever $n\notin \Lambda$. That is,

\begin{definition}
\[
H^2(\mathbb{D}, \Lambda):=\left\{f\in H^2(\mathbb{D}):\, f(z)=\sum_{n=1}^{\infty} c_{\lambda_n} z^{\lambda_n}\right\}.
\]
\end{definition}

\begin{example}
If $\Lambda=\{n^2\}_{n=1}^{\infty}$, then $f(z)=\sum_{n=1}^{\infty} \frac{1}{n} z^{n^2}$ belongs to
$H^2 (\mathbb{D}, \Lambda)$.
\end{example}

We aim to find a condition which is both necessary and sufficient so that an analytic function
in the disk $\mathbb{D}$ belongs to the space $H^2(\mathbb{D}, \Lambda)$. This condition will be in terms of

(i) the space $\overline{M_{\Lambda}}$

and

(ii) of the unique biorthogonal family $r_{\Lambda}$ to the system $M_{\Lambda}$ in $\overline{M_{\Lambda}}$.

We prove the following.
\begin{theorem}\label{HS}
Let $\Lambda=\{\lambda_n\}_{n=1}^{\infty}\subset\mathbb{N}$ such that $\Lambda$
satisfies $(\ref{LKcondition})$.
Let $r_{\Lambda}$ be the unique biorthogonal family of $M_{\Lambda}$ in the space $\overline{M_{\Lambda}}$.
Then
\[
f\in H^2 (\mathbb{D}, \Lambda)
\]
if and only if
\begin{equation}\label{frame}
f\in\overline{M_{\Lambda}}\qquad \text{and}\qquad \sum_{n=1}^{\infty}\left| \langle f, r_n\rangle \right|^2<\infty.
\end{equation}
\end{theorem}

\section{Auxiliary lemmas}
\setcounter{equation}{0}

\begin{lemma}\label{L2}
Let $\Lambda=\{\lambda_n\}_{n=1}^{\infty}\subset\mathbb{N}$ such that $\Lambda$
satisfies $(\ref{LKcondition})$.
If $f\in H^2 (D, \Lambda)$ then for every fixed $\theta\in [0,2\pi]$, the function $f(t e^{i\theta})$ belongs
to the space $L^2 (0, 1)$, that is,
\[
\int_{0}^{1} |f(te^{i\theta}|^2\, dt<\infty.
\]
\end{lemma}
\begin{proof}

Since $\sum_{n=1}^{\infty}\lambda_n^{-1}<\infty$, then $n/\lambda_n\to 0$ as $n\to\infty$, and hence $\sum_{n=1}^{\infty} t^{2\lambda_n}<\infty$ for any $0<t<1$. Since $f\in H^2 (D, \Lambda)$ then $f(z)=\sum_{n=1}^{\infty} c_{\lambda_n} z^{\lambda_n}$ thus
$\sum_{n=1}^{\infty} |c_{\lambda_n}|^2<\infty$. By the Cauchy-Schwarz inequality we then have
\[
\sum_{n=1}^{\infty} \left|c_{\lambda_n}\cdot t^{\lambda_n}\right|^2\le
\left(\sum_{n=1}^{\infty} |c_{\lambda_n}|^2\right) \cdot \left(\sum_{n=1}^{\infty} t^{2\lambda_n}\right).
\]

Therefore, there is some $M>0$ so that for all $0<\rho<1$, we have

\begin{eqnarray*}
\int_{0}^{\rho} |f(te^{i\theta}|^2\, dt=
\int_{0}^{\rho} \left|\sum_{n=1}^{\infty} c_{\lambda_n} t^{\lambda_n} e^{i\theta \lambda_n}\right|^2\, dt & \le &
\int_{0}^{\rho} \left(\sum_{n=1}^{\infty} |c_{\lambda_n}|\cdot  t^{\lambda_n}\right)^2\, dt\\ & \le &
\left(\sum_{n=1}^{\infty} |c_{\lambda_n}|^2\right)\cdot
\int_{0}^{\rho} \sum_{n=1}^{\infty} t^{2\lambda_n}\, dt\\ & = &
\left(\sum_{n=1}^{\infty} |c_{\lambda_n}|^2\right)\cdot  \left(\sum_{n=1}^{\infty}  \int_{0}^{\rho} t^{2\lambda_n}\, dt\right)\\ & \le &
\left(\sum_{n=1}^{\infty} |c_{\lambda_n}|^2\right)\cdot  \left(\sum_{n=1}^{\infty}  \frac{1}{2\lambda_n +1}\right)\\ & \le & M.
\end{eqnarray*}

This uniform upper bound for $\int_{0}^{\rho} |f(te^{i\theta}|^2\, dt$, $0<\rho<1$, means that
\[
\int_{0}^{1} |f(te^{i\theta}|^2\, dt\le M.
\]
\end{proof}

The following result is known \cite[Corollary 6.2.4]{Gurariy} however for the sake of completeness we present below a proof.
\begin{lemma}\label{converse}
Let $\Lambda=\{\lambda_n\}_{n=1}^{\infty}$ be a strictly increasing sequence of positive real numbers
that satisfies $(\ref{LKcondition})$.
Suppose that $f\in L^2 (0,1)$ and $f(t)=\sum_{n=1}^{\infty} c_n t^{\lambda_n}$ for $t\in [0,1)$ with the series converging uniformly on compact
subsets of $[0, 1)$. Then $f$ belongs to the space $\overline{M_{\Lambda}}$.
\end{lemma}

\begin{proof}

By applying Fatou's Lemma and changing variables we get

\begin{eqnarray*}
\int_{0}^{1}|f(t)|^2\,dt\le \liminf_{\rho\to 1^-}\int_{0}^{1}|f(\rho t)|^2\,dt\le \limsup_{\rho\to 1^-}\int_{0}^{1}|f(\rho t)|^2\,dt & = &
\limsup_{\rho\to 1^-}\int_{0}^{\rho}|f(u)|^2\, \frac{du}{\rho}\\ & \le &
\limsup_{\rho\to 1^-}\frac{1}{\rho}\cdot \int_{0}^{1}|f(u)|^2\, du \\ & = &
\int_{0}^{1}|f(t)|^2\,dt.
\end{eqnarray*}

This means that $\lim_{\rho\to 1^-}\int_{0}^{1}|f(\rho t)|^2\,dt$ exists and

\[
\lim_{\rho\to 1^-}\int_{0}^{1}|f(\rho t)|^2\, dt = \int_{0}^{1}|f(t)|^2\, dt.
\]
Hence we have
\[
\lim_{\rho\to 1^-}\int_{0}^{1}|f(\rho t)-f(t)|^2\, dt = 0.
\]
Therefore, for every $\epsilon>0$ there is some $0<\delta_{\epsilon}<1$ so that for all $\rho\in (\delta_{\epsilon}, 1)$ one has
\begin{equation}\label{epsilon}
\int_{0}^{1}|f(\rho t)-f(t)|^2\, dt < \epsilon.
\end{equation}

Fix now such an $\epsilon$, its $\delta_{\epsilon}$, as well as some $\rho \in (\delta_{\epsilon}, 1)$. Then the series

\[
f(\rho t) = \sum_{n=1}^{\infty} c_{n} \cdot \rho^{\lambda_n}\cdot t^{\lambda_n}
\]
converges uniformly on the interval $[0,1]$.

Thus there is some positive integer $N$, which depends on $\epsilon$ and $\rho$, so that
\[
\int_{0}^{1}\left| f(\rho t) - \sum_{n=1}^{N} c_{n} \cdot \rho^{\lambda_n}\cdot t^{\lambda_n}\right|^2\, dt<\epsilon.
\]
Combined with $(\ref{epsilon})$ and applying the Minkowski inequality shows that
\[
\int_{0}^{1}\left| f(t) - \sum_{n=1}^{N} c_{n} \cdot \rho^{\lambda_n}\cdot t^{\lambda_n}\right|^2\, dt<2\epsilon.
\]
This last inequality means that $f$ belongs to the space $\overline{M_{\Lambda}}$.
\end{proof}

\begin{lemma}\label{Bio}
Let $\Lambda=\{\lambda_n\}_{n=1}^{\infty}$ be a strictly increasing sequence of positive real numbers
that satisfies $(\ref{LKcondition})$. Then there exists a unique family of functions
\[
r_{\Lambda}=\{r_{n}:\,\, n\in\mathbb{N}\}\subset \overline{M_{\Lambda}}
\]
so that it is biorthogonal to $M_{\Lambda}$, and such that
for every $\epsilon>0$ there is a constant $m_{\epsilon}>0$, independent of $n\in\mathbb{N}$,
but depending on $\Lambda$, so that
\begin{equation}\label{rnkbound}
||r_{n}|| \le  m_{\epsilon}(1+\epsilon)^{\lambda_n},\qquad \forall\,\, n\in\mathbb{N}.
\end{equation}
\end{lemma}
\begin{proof}
Recall that $e_{n}(t)=t^{\lambda_n}$ and let $M_{\Lambda_{n}}:=M_{\Lambda}\setminus e_{n}$,
that is $M_{\Lambda_{n}}$ is the system $M_{\Lambda}$ excluding the function $t^{\lambda_n}$.
Denote by $D_{n}$ the $distance$ between $e_{n}$
and the closed span of $M_{\Lambda_{n}}$ in $L^2(0,1)$, that is
\[
D_{n}:=\inf_{g\in \overline{\text{span}} (M_{\Lambda_{n}})} ||e_{n}-g||_{L^2 (0,1)}.
\]
Consider also the space $L^2 (a,1)$ for some $0<a<1$ and denote by $D_{n,a}$ the distance between $e_{n}$
and the closed span of $M_{\Lambda_{n}}$ in $L^2(a,1)$.
It was shown in \cite[relation (1.9)]{LK} that for every $\epsilon>0$, there is a positive constant $m_{\epsilon}$
which depends on $a$ and on $\Lambda$ but not on $n\in\mathbb{N}$, so that
$D_{n,a}\ge m_{\epsilon}\cdot (1-\epsilon)^{\lambda_n}$. If a function $g$ belongs to the closed span of $M_{\Lambda_{n}}$ in $L^2(0,1)$,
then clearly $g$ belongs to the closed span of $M_{\Lambda_{n}}$ in $L^2(a,1)$ and $||e_{n}-g||_{L^2 (0,1)}\ge ||e_{n}-g||_{L^2 (a,1)}$.
Thus $D_n\ge D_{n,a}$, hence we conclude that
\begin{equation}\label{distanceresult}
D_n\ge m_{\epsilon} (1-\epsilon)^{\lambda_n}.
\end{equation}
Now, from the $Closed\,\, Point\,\, Theorem$ in Hilbert spaces,
and since $L^2(0,1)$ is such a space, endowed with the inner product
$\langle f,g \rangle:=\int_{0}^{1} f(t)\overline{g(t)}\, dt$, it  follows that there exists a unique element
in $\overline{\text{span}}(M_{\Lambda_{n}})$ in $L^2(0,1)$, that we denote by $\varphi_{n}$,
so that
\[
||e_{n}-\varphi_{n}||_{L^2 (0,1)}=
\inf_{g\in \overline{\text{span}}(M_{\Lambda_{n}})}||e_{n}-g||_{L^2 (0,1)}=D_{n}.
\]
The function $e_{n}-\varphi_{n}$ is orthogonal to all the elements of the closed span of $M_{\Lambda_{n}}$
in $L^2 (0,1)$, thus to $\varphi_{n}$ as well. Hence
\[
\langle e_{n}-\varphi_{n}, e_{n}-\varphi_{n} \rangle=\langle e_{n}-\varphi_{n}, e_{n}\rangle.
\]
Clearly one has
\[
(D_{n})^2=\langle e_{n}-\varphi_{n}, e_{n}\rangle.
\]
Next, let
\[
r_{n}(t):=\frac{e_{n}(t)-\varphi_{n}(t)}{(D_{n}) ^2}.
\]
One now has $\langle r_{n}, e_{n}\rangle=1$ and $\langle r_{n}, e_{m}\rangle=0$ for all $m\in\mathbb{N}\setminus n$.
We just proved that $\{r_{n}:\,\, n\in\mathbb{N}\}$ is biorthogonal  to the system $M_{\Lambda}$. Moreover,
since $\varphi_{n}\in\overline{\text{span}}(M_{\Lambda_{n}})$ in $L^2 (0,1)$ then
$r_{n}\in\overline{M_{\Lambda}}$. The completeness of $M_{\Lambda}$ in $\overline{M_{\Lambda}}$ implies that $r_{\Lambda}$ is the only
biorthogonal family to $M_{\Lambda}$ in the space $\overline{M_{\Lambda}}$.
Finally, $(\ref{rnkbound})$ follows from the definition of $r_n$ and relation $(\ref{distanceresult})$.
\end{proof}

\section{Proof of Theorems $\ref{biorthogonalsystem}$ and $\ref{Spectral Synthesis}$}
\setcounter{equation}{0}

\subsection{Proof of Theorem $\ref{biorthogonalsystem}$}

($\bf Proof$ of I)

From the ``Clarkson-Erd\H{o}s-Schwartz Phenomenon''
we know that any $f\in\overline{M_{\Lambda}}$ extends analytically in the slit disk $\mathbb{D^*}=D\setminus (-1,0)$ and admits the series representation $f(z)=\sum_{n=1}^{\infty} c_n z^{\lambda_n}$ converging uniformly on compact subsets of $\mathbb{D^*}$.
We will show below that
$c_{n}=\langle f, r_{n} \rangle$ for all $n\in\mathbb{N}$.

Fix any $n\in\mathbb{N}$ and let
\[
f_n (t) := f(t) - c_n t^{\lambda_n}.
\]

We claim that
\begin{equation}\label{zerozero1}
\int_{0}^{1} \overline{r_{n}(t)}\cdot f_n (t) \, dt=0.
\end{equation}

Indeed, since $f\in L^2(0,1)$ then $f_n\in L^2(0,1)$ as well. But
$f_n(t)=\sum_{k\not= n} c_k t^{\lambda_k}$, thus it follows from Lemma $\ref{converse}$ that $f_n$ belongs
to the closed span of the system
\[
M_{\Lambda_n}=M_{\Lambda}\setminus e_n.
\]
in $L^2(0,1)$. Hence, for every $\epsilon>0$, there is a function $f_{n,\epsilon}$ in the span of $M_{\Lambda_n}$
so that $||f_n-f_{n,\epsilon}||_{L^2(0,1)}<\epsilon$. Due to the biorthogonality we have
\[
\int_{0}^{1} \overline{r_{n}(t)}\cdot f_{n,\epsilon}(t)\, dt=0.
\]
Combining this with the Cauchy-Schwarz inequality we get
\begin{eqnarray*}
\left|\int_{0}^{1} \overline{r_{n}(t)}\cdot f_n(t)\, dt\right| & = &
\left|\int_{0}^{1} \overline{r_{n}(t)}\cdot (f_n (t)-f_{n, \epsilon}(t))\, dt\right|\\
& \le & \epsilon\cdot ||r_{n}||_{L^2(0,1)}.
\end{eqnarray*}
The arbitrary choice of $\epsilon$ implies that $(\ref{zerozero1})$ is true.

Since $\langle e_n, r_n\rangle =1$ we get
\begin{eqnarray*}
\langle f, r_{n} \rangle & = & \int_{0}^{1}\overline{r_{n}(t)} \cdot  c_n t^{\lambda_n}\, dt  +
\int_{0}^{1}\overline{r_{n}(t)} \cdot f_n (t)\, dt \\
& = & c_n.
\end{eqnarray*}

($\bf Proof$ of II)

Clearly $f(x)$ can be written uniquely as
\[
f(x)=g(x)+h(x)
\]
where $g$ belongs to the space $\overline{M_{\Lambda}}$ and $h$ belongs to the orthogonal
complement of $\overline{M_{\Lambda}}$ in $L^2 (0, 1)$. Thus $\langle h, r_n\rangle=0$ for all $r_n\in r_{\Lambda}$
and since $h=f-g$, we then have
\[
\langle f, r_n\rangle=\langle g, r_n\rangle\qquad \text{for all}\quad n\in\mathbb{N}.
\]
Since $g\in \overline{M_{\Lambda}}$ then from $(\ref{representationf})$ we have
\[
g(x)=\sum_{n=1}^{\infty}\langle g, r_n \rangle \cdot x^{\lambda_n}\qquad \text{for all}\quad x\in [0,1).
\]
Combining the above shows that
\[
\sum_{n=1}^{\infty}\langle f, r_n \rangle\cdot  x^{\lambda_n}\qquad
\text{belongs to the space $\overline{M_{\Lambda}}$}.
\]

($\bf Proof$ of III)

Denote by $S_{\Lambda}$ the closed span of the family $r_{\Lambda}$ in $L^2 (0,1)$ hence $S_{\Lambda}$ is a subspace
of $\overline{M_{\Lambda}}$. Let $S_{\Lambda}^{\perp}$ be the orthogonal complement of $S_{\Lambda}$ in $\overline{M_{\Lambda}}$, that is,

\[
S_{\Lambda}^{\perp}=\{f\in \overline{M_{\Lambda}}:\,\, \langle f, g \rangle =
0\quad \text{for\,\, all}\,\, g\in S_{\Lambda}\}.
\]
If $f\in S_{\Lambda}^{\perp}$ then $\langle f, r_n \rangle = 0$ for all $r_n\in r_{\Lambda}$
and $f\in \overline{M_{\Lambda}}$. But any $f\in \overline{M_{\Lambda}}$ admits the representation $(\ref{representationf})$ thus
$f$ is the zero function. Hence $S_{\Lambda}=\overline{M_{\Lambda}}$.

($\bf Proof$ of IV)

We write the set $\mathbb{N}$ as an arbitrary disjoint union of two sets $N_1$ and $N_2$ and
we must show that the closed span of the mixed system
\[
M_{\Lambda_{1,2}}:=\{e_{n}:\,\, n\in N_1\}\cup\{r_{n}:\,\, n\in N_2\},
\]
in $L^2 (0, 1)$ is equal to $\overline{M_{\Lambda}}$.

Let $\overline{M_{\Lambda_{1,2}}}$ be the closed span of $M_{\Lambda_{1,2}}$ in $L^2 (0, 1)$ and
let $\overline{M_{\Lambda_{1,2}}}^{\perp}$ be the orthogonal complement of $\overline{M_{\Lambda_{1,2}}}$
in $\overline{M_{\Lambda}}$, that is
\[
\overline{M_{\Lambda_{1,2}}}^{\perp}=\{f\in \overline{M_{\Lambda}}:\,\, \langle f, g \rangle =
0\quad \text{for\,\, all}\,\, g\in \overline{M_{\Lambda_{1,2}}}\}.
\]
It follows from $(\ref{representationf})$ that if $f\in \overline{M_{\Lambda_{1,2}}}^{\perp}\subset M_{\Lambda}$,
hence $f\in\overline{M_{\Lambda}}$ also, then
\[
f(t)=\sum_{n=1}^{\infty} \langle f, r_{n} \rangle \cdot t^{\lambda_n},\quad \text{on}\quad (0,1).
\]
and $\langle f, r_{n} \rangle =0$ for all $n\in N_2$.
Thus,
\[
f(t)=\sum_{n\in N_1}\langle f, r_{n} \rangle \cdot t^{\lambda_n},
\quad \text{on}\quad (0,1).
\]
It then follows from Lemma $\ref{converse}$ that $f$ belongs to the closed span of the system
\[
M_{\Lambda, N_1}:=\{e_{n}:\,\, n\in N_1\}
\]
in $L^2 (0, 1)$. So, for every $\epsilon>0$ there is a function $g_{\epsilon}$ in $\text{span}(M_{\Lambda, N_1})$ so that
$||f-g_{\epsilon}||_{L^2(0,1)}<\epsilon$.
Since $f\in \overline{M_{\Lambda_{1,2}}}^{\perp}$ then $\langle f, e_{n} \rangle =0$ for all $n\in N_1$, thus
$\langle f, g_{\epsilon} \rangle=0$. Therefore,
\[
||f||^2_{L^2(0,1)}=\langle f, f \rangle = \langle f, f-g_{\epsilon}\rangle \le ||f||_{L^2(0,1)}\cdot ||f-g_{\epsilon}||_{L^2(0,1)}
\le ||f||_{L^2(0,1)}\cdot \epsilon.
\]
Hence
\[
||f||_{L^2(0,1)}\le \epsilon.
\]
Clearly this means that $f$ is the zero function hence $\overline{M_{\Lambda_{1,2}}}^{\perp}=\{\mathbf{0}\}$.
Thus $\overline{M_{\Lambda_{1,2}}}=\overline{M_{\Lambda}}$, concluding that $M_{\Lambda}$
is a strong Markushevich basis for $\overline{M_{\Lambda}}$.

\subsection{Proof of Theorem $\ref{Spectral Synthesis}$}

The proof follows along the lines of the results in \cite[Section 4]{Zikkos2024JMAA}.

We will show below that the operator $T$ in $(\ref{operator})$ has properties
as those stated in Theorem $A$. Combined with the fact that the system
$M_{\Lambda}$ is hereditarily complete in $\overline{M_{\Lambda}}$ shows that Theorem $\ref{Spectral Synthesis}$ is true.

We will prove that
\begin{enumerate}
\item The operator $T$ is compact.
\item $\{ u_k\}_{k=1}^{\infty}$ are eigenvalues  of $T$ and $\{ e_k\}_{k=1}^{\infty}$ are the corresponding eigenvectors.
\item $\{\overline{u_k}\}$ are eigenvalues  of $T^*$ and $\{ r_k\}_{k=1}^{\infty}$ are the corresponding eigenvectors.
\item The kernel of $T$ is trivial.
\item The spectrum of $T$ is $\displaystyle{\{ 0\} \cup \{ u_k\}_{k=1}^{\infty}}$.
\item Each eigenvalue of $T$ is simple.
\item The operator $T$ is not normal.
\end{enumerate}

1. First we prove that $T$ is a bounded operator on $\overline{M_{\Lambda}}$.
Let $f$ be a function in the space $\overline{M_{\Lambda}}$. Then from Theorem $\ref{biorthogonalsystem}$ we have
\[
f(z)=\sum_{n=1}^{\infty} \langle f , r_n\rangle\cdot z^{\lambda_n},
\]
with the series converging uniformly on compact subsets of the slit disk $\mathbb{D^*}$.
From relation $(\ref{rnkbound})$, for every $\epsilon>0$ there exists some $m_{\epsilon}>0$, independent of $n\in\mathbb{N}$, so that
\begin{equation}\label{rnkbounda}
|\langle f , r_n\rangle|\le ||f||_{L^2 (0,1)}\cdot m_{\epsilon}(1+\epsilon)^{\lambda_n}.
\end{equation}
Consider now the sequence $\{u_n\}_{n=1}^{\infty}$ as in $(\ref{un})$,
satisfying $|u_n|\le  \rho^{\lambda_n}$ for some $\rho\in (0,1)$, hence $u_n\to 0$ as $n\to\infty$.
Fix some positive $\epsilon$ so that  $\epsilon<\frac{1-\rho}{2\rho}$ thus $1+\epsilon<\frac{\rho +1}{2\rho}$.
Together with $(\ref{rnkbounda})$ shows that

\begin{eqnarray}
|\langle f , r_n\rangle\cdot u_n| & \le & ||f||_{L^2 (0,1)}\cdot m_{\epsilon}(1+\epsilon)^{\lambda_n}\cdot \rho^{\lambda_n}\nonumber\\
& < & ||f||_{L^2 (0,1)}\cdot m_{\epsilon}\cdot \left(\frac{\rho +1}{2}\right)^{\lambda_n}.\label{upperbound}
\end{eqnarray}

One now deduces from $(\ref{upperbound})$ that
\[
T(f(z))=\sum_{n=1}^{\infty} \langle f , r_n\rangle\cdot u_n\cdot  z^{\lambda_n}.
\]
defines an analytic function in the slit disk
\[
\left\{z:\,\, |z|<\frac{2}{\rho +1}\right\}\setminus \left(-\frac{2}{\rho +1}, 0\right),
\]
converging uniformly on its compact subsets. Since $\frac{2}{\rho +1}>1$ then the series $T(f(z))$
converges uniformly on the closed interval $[0,1]$. This means that

$(i)$ $T(f(z))$ belongs to the space $\overline{M_{\Lambda}}$.

$(ii)$ The series $T(f)$ converges in the $L^2(0,1)$ norm.

It also follows from $(\ref{upperbound})$ that there exists some $B>0$ so that
\[
||T(f)||_{L^2(0,1)}\le B||f||_{L^2(0,1)}\quad \text{for\,\, all}\quad f\in \overline{M_{\Lambda}}.
\]
Therefore, $T: \overline{M_{\Lambda}} \to \overline{M_{\Lambda}}$ defines a bounded linear operator
and we denote by $T^*$  its adjoint operator.

\smallskip

Next we prove that $T$ is a compact operator on $\overline{M_{\Lambda}}$.
If $L$ is a bounded operator on $\overline{M_{\Lambda}}$, we use $\| L \|$ to
denote the operator norm of $L$ which is the supremum of the set
$ \{ \| L(f)\|_{L^2(0,1)}: f\in \overline{M_{\Lambda}}, \| f\|_{L^2(0,1)}=1 \}$.

Define now finite rank operators $T_m$ on $\overline{M_{\Lambda}}$ by
\[
T_m(f)(x):=\sum_{n=1}^m \langle f, r_n \rangle u_n x^n\qquad m=1,2,\dots
\]
Let $f$ be a unit vector in $\overline{M_{\Lambda}}$. It is easy to see that
\[
\|(T-T_m)(f)\|_{L^2(0,1)}\leq \sum_{n=m+1}^\infty \| \langle f, r_n \rangle u_n e_n \|_{L^2(0,1)}\qquad \text{where}\quad e_n(x)=x^{\lambda_n}.
\]

From $(\ref{upperbound})$ we get $\| \langle f, r_n \rangle u_n e_n  \|_{L^2(0,1}\leq
m_\epsilon\cdot  \left(\frac{\rho +1}{2}\right)^{\lambda_n}\cdot \| e_n \|_{L^2(0,1)}.$
Clearly $\| e_n \|_{L^2(0,1)}\le 1$
thus  $\| \langle f, r_n \rangle u_n e_n  \|_{L^2(0,1)}\leq m_\epsilon\cdot  \left(\frac{\rho +1}{2}\right)^{\lambda_n}$.
Therefore,
\[
\|(T-T_m)(f)\|_{L^2(0, 1)}\leq m_\epsilon  \sum_{n=m+1}^\infty \left(\frac{\rho +1}{2}\right)^{\lambda_n}.
\]
Since $\rho<1$  then $\|T-T_m\|$ tends to zero as $m$ tends to infinity.
Thus, the finite rank operators $\{ T_m\}_{m=1}^{\infty}$ converge to $T$ in the uniform operator topology, hence
$T$ is compact.

2. Due to biorthogonality from $(\ref{operator})$ we get
\[
T(e_k)=u_k e_k.
\]

3.  For fixed $k\in\mathbb{N}$ and any $n\in\mathbb{N}$ we have
\begin{eqnarray}
\langle T^*r_k-\overline{u_k} r_k, e_n\rangle & = & -\overline{u_k}\langle r_k, e_n\rangle +\langle r_k, Te_n\rangle\nonumber\\
& = & -\overline{u_k}\langle r_k, e_n\rangle + \langle r_k, u_n e_n\rangle\nonumber\\
& = & -\overline{u_k}\langle r_k, e_n\rangle + \overline{u_n}\langle r_k, e_n\rangle. \label{zero}
\end{eqnarray}
Clearly $(\ref{zero})$ equals zero if $n=k$ and due to the biorthogonality the same holds for all $n\not= k$. Hence
$\langle T^*r_k-\overline{u_k} r_k, e_n\rangle=0$ for all $n\in\mathbb{N}$. The completeness of $M_{\Lambda}$ in $\overline{M_{\Lambda}}$
means that
\[
T^*r_k-\overline{u_k}\cdot r_k=\mathbf{0}.
\]

4. For any $n\in\mathbb{N}$ we have
\[
\langle Tf, r_n\rangle =\langle f, T^*r_n\rangle = \langle f, \overline{u_n} r_n\rangle.
\]
If $Tf=0$ we get $\langle f, r_n\rangle =0$ for all $n\in\mathbb{N}$. It follows that $f$ is the zero function since the family $\{r_n\}_{n=1}^{\infty}$ is complete in $\overline{M_{\Lambda}}$. Thus the kernel of $T$ is trivial.\\

\smallskip

5. Suppose now that $Tf=\lambda f$ for some $\lambda\notin\{u_n\}_{n=1}^{\infty}$ and $f\not=\mathbf{0}$. Then,
\begin{align}
\lambda\langle f, r_n\rangle& =  \langle Tf, r_n\rangle, \notag \\
& =\langle f, T^*r_n\rangle \notag, \\
&= \langle f, \overline{u_n} r_n\rangle, \notag \\
&=u_n\langle f, r_n\rangle\qquad \text{for all}\quad n\in\mathbb{N}.\nonumber
\end{align}
Hence
\[
(\lambda-u_n)\cdot \langle f, r_n\rangle=0\qquad \text{for all}\quad n\in\mathbb{N}.
\]
Since $\lambda\notin\{u_n\}_{n=1}^{\infty}$ then $\langle f, r_n\rangle=0$ for all $n\in\mathbb{N}$ and the completeness of
the family $\{r_n\}_{n=1}^{\infty}$ in $\overline{M_{\Lambda}}$ means that $f=\mathbf{0}$, a contradiction.
Thus the only non-zero eigenvalues are the $u_k$'s and
the compactness of $T$ means that its spectrum is
${\{ 0\} \cup \{ u_k\}_{k=1}^{\infty}}$.

6. Next, suppose that $Tf=u_k f$ for some $u_k\in\{u_n\}_{n=1}^{\infty}$ and some $f\in \overline{M_{\Lambda}}$.
Similarly to part 5 above, we get

\[
(u_k-u_n)\langle f , r_n\rangle = 0 \qquad \text{for all}\quad n\in\mathbb{N}.
\]
If $n\not= k$ then $u_n\not= u_k$, therefore $\langle f , r_n\rangle = 0$ for all $n\not= k$. However,
\[
f(t)=\sum_{n=1}^{\infty} \langle f , r_n\rangle\cdot  e_n(t),
\]
hence we conclude that $f(t)=\langle f , r_k\rangle e_k$, meaning that the eigenvalue $u_k$ is simple.

7. If $T$ is a normal operator, then
\[
TT^*(e_k)=T^*T(e_k)\qquad \text{for\,\,all}\quad k\in\mathbb{N}.
\]
But for a normal operator any two eigenvectors corresponding to different eigenvalues are orthogonal.
However, the set of eigenvectors $e_n=t^{\lambda_n}$ of $T$ are not orthogonal, hence $T$ is not normal.

This completes the proof of Theorem $\ref{Spectral Synthesis}$.

\section{Proof of Theorem $\ref{HS}$}
\setcounter{equation}{0}

Suppose first that $(\ref{frame})$ holds. Since $f$ belongs to the space $\overline{M_{\Lambda}}$ then
by Theorem $\ref{biorthogonalsystem}$ $f$ extends analytically in the unit disk $\mathbb{D}$ and is represented there as
$f(z)=\sum \langle f, r_n\rangle z^{\lambda_n}$. It then follows from the right relation in $(\ref{frame})$
that $f$ belongs to the space $H^2 (\mathbb{D}, \Lambda)$.

On the other hand, suppose now that $f$ belongs to $H^2 (\mathbb{D}, \Lambda)$ thus $f(z)=\sum_{n=1}^{\infty} c_n z^{\lambda_n}$ and
$\sum_{n=1}^{\infty} |c_n|^2<\infty$. From Lemma $\ref{L2}$ we know that $f$ belongs to the space $L^2 (0,1)$.
It then follows from Lemma $\ref{converse}$ that $f$ belongs to the space $\overline{M_{\Lambda}}$, thus from $(\ref{representationf})$
we have the series representation $f(z)=\sum_{n=1}^{\infty}\langle f, r_n\rangle z^{\lambda_n}$.
But the coefficients of a power series are unique,
thus $c_n=\langle f, r_n\rangle$. Since $\sum_{n=1}^{\infty}|c_n|^2<\infty$ then one has
$\sum_{n=1}^{\infty}|\langle f, r_n\rangle |^2<\infty$. Therefore we conclude that $(\ref{frame})$ holds.

Thus $(\ref{frame})$ is both a necessary and a sufficient condition in order for an analytic function in the disk $\mathbb{D}$ to belong to
$H^2 (\mathbb{D}, \Lambda)$. This finishes the proof of Theorem $\ref{HS}$.

\section{Acknowledgments}
I would like to thank Professor W. Lusky for private communication.

\end{document}